\documentclass[11pt]{amsart}
\usepackage{amsmath, amssymb}
\usepackage{hyperref}
\usepackage{tikz}

\newcounter{my_enumerate_counter}
\newcommand{\pushcounter}{\setcounter{my_enumerate_counter}{\value{enumi}}}
\newcommand{\popcounter}{\setcounter{enumi}{\value{my_enumerate_counter}}}

\DeclareMathOperator{\bfTrep}{\mathbf T_{rep}} 
\DeclareMathOperator{\calLrep}{\calL_{rep}}

\DeclareMathOperator{\id}{id}
\newcommand{\cU}{\mathcal U}

\newcommand{\bt}{\mathbf t}

\newcommand{\bbF}{{\mathbb F}}

\newcommand{\bbZ}{{\mathbb Z}}

\newcommand{\bbN}{{\mathbb N}}
\newcommand{\bbC}{\mathbb C}

\newcommand{\calL}{\mathcal L}
\newcommand{\cM}{{\mathcal M}}

\newcommand{\rs}{\restriction}

\newcommand{\cK}{\mathcal K}

\newcommand{\e}{\varepsilon}

\newtheorem{thm}{Theorem}[section]
\newtheorem{theorem}{Theorem}

\newtheorem{corollary}[theorem]{Corollary}

\newtheorem{question}[thm]{Question}

\newtheorem{lemma}[thm]{Lemma}

\newtheorem{prop}[thm]{Proposition}

\DeclareMathOperator{\dist}{dist}

\theoremstyle{definition}

\newcommand{\dminus}{ 
\buildrel\textstyle\ .\over{\hbox{ 
\vrule height3pt depth0pt width0pt}{\smash-} 
}}

\newcommand{\cstar}{$\mathrm{C}^*$}
\newcommand{\cst}{\mathrm{C}^*}

\newcommand{\prB}{\prod_{\cU} B_j}
\newcommand{\prH}{\prod_{\cU} H_j}

\newcommand{\WOT}[1]{\overline{#1}^{\rm{WOT}}}

\title{A new bicommutant theorem} 
\author{I.\ Farah}
\address{Department of Mathematics and Statistics,
York University,
4700 Keele Street,
North York, Ontario, Canada, M3J
1P3}
\email{ifarah@mathstat.yorku.ca}
\urladdr{http://www.math.yorku.ca/$\sim$ifarah}

\date{\today}

\begin{document}

\begin{abstract} We prove an  analogue of Voiculescu's theorem: 
Relative bicommutant of a  separable unital subalgebra $A$ of an ultraproduct of simple unital 
\cstar-algebras is equal to $A$. 
\end{abstract} 
\maketitle

Ultrapowers\footnote{Throughout  $\cU$ denotes a nonprincipal ultrafilter on $\bbN$.}
 $A^{\cU}$ of separable algebras are, being subject to well-developed model-theoretic methods,  reasonably  well-understood    
 (see e.g. \cite[Theorem~1.2]{FaHaRoTi:Relative} and \S\ref{S.SR}). 
 Since the early 1970s and the influential 
 work of McDuff and Connes
  central sequence algebras $A'\cap A^{\cU}$ play an even more important
role than ultrapowers in the classification of II$_1$ factors and (more recently)  \cstar-algebras. 
While they do not have a well-studied abstract  
analogue, in \cite[Theorem~1]{FaHaRoTi:Relative} it was shown that  
the central sequence algebra of a strongly self-absorbing algebra (\cite{ToWi:Strongly}) is isomorphic to its ultrapower if 
the Continuum Hypothesis holds.  
Relative commutants  $B'\cap D^{\cU}$  
of separable subalgebras of ultrapowers of strongly self-absorbing \cstar-algebras
play an increasingly important role 
in classification program for separable \cstar-algebras
(\cite[\S 3]{matui2014decomposition}, \cite{bosa2015covering}; see also  \cite{tikuisis2015quasidiagonality}, \cite{Win:QDQ}). 
In the present note we make a step towards better  understanding of 
these algebras. 
  
 \cstar-algebra is \emph{primitive} if it has  representation that is both faithful and irreducible. 
We prove an 
 analogue of the well-known consequence of Voiculescu's 
 theorem (\cite[Corollary~1.9]{voiculescu1976non}) and von Neumann's bicommutant theorem (\cite[\S I.9.1.2]{Black:Operator}).

\begin{theorem} \label{T1} 
Assume $\prB$ is an ultraproduct of  primitive 
 \cstar-algebras   and 
$A$ is a separable  \cstar-subalgebra. In addition assume $A$ is a unital subalgebra if $\prB$ is unital. Then 
(with $\WOT {A}$ computed in the 
ultraproduct of faithful irreducible representations of $B_j$s)
\[
A=\biggl(A'\cap \prB\biggr)'
=\WOT{A}\cap \prB. 
\] 
\end{theorem}

A slightly weaker version of the following corollary to Theorem~\ref{T1} (stated here 
with Aaron Tikuisis's kind permission) was originally proved by using very different 
methods ($Z(A)$ denotes the center of $A$). 

\begin{corollary}[Farah--Tikuisis, 2015] \label{C1}
Assume $\prB$ is an ultraproduct of simple unital  \cstar-algebras   and 
$A$ is a separable unital subalgebra. 
Then  $Z(A'\cap \prB)=Z(A)$. \qed
\end{corollary}

At least two open 
problems are  concerned with bicommutants of separable subalgebras of massive operator algebras. 
As is well-known, central sequence algebras $M'\cap M^{\cU}$  of~II$_1$ factors 
in tracial ultrapowers behave differently from the central sequence algebras  of \cstar-algebras. 
For a  II$_1$ factor $M$ with  separable predual the central sequence algebra 
$M'\cap M^{\cU}$ can be abelian or even trivial.  
Popa conjectured that 
if $P$ is a separable subalgebra of an ultraproduct of  II$_1$ factors  
then $(P'\cap \prod_{\cU}N_i)'=P$ implies $P$ is amenable (\cite[Conjecture~2.3.1]{popa2014independence}). 
In the domain of \cstar-algebras, 
Pedersen 
asked (\cite[Remark~10.11]{Pede:Corona}) whether the following variant of 
Theorem~\ref{T1} is true: If  the corona $M(B)/B$ of a $\sigma$-unital \cstar-algebra $B$ 
is simple and $A$ is a separable unital subalgebra, is  $(A'\cap M(B)/B)'=A$? 
(For the connection between ultraproducts and coronas see 
the last paragraph of \S\ref{S.Day}.)

The proof   of Theorem~\ref{T1} 
uses logic of metric structures (\cite{BYBHU}, \cite{FaHaSh:Model2}) 
and  an analysis of the 
interplay between \cstar-algebra $B$ and its second dual~$B^{**}$. 

\subsection*{Acknowledgments} 
Theorem~\ref{T1}
 was inspired by conversations with Aaron Tikuisis and Stuart White and I use this opportunity to thank them. 
 Corollary~\ref{C1} was proved during a very inspirational  visit to the University of Aberdeen in July 2015. 
 I am indebted to Aaron Tikuisis for warm hospitality and stimulating discussions 
and to the   London Mathematical Society  for funding my  visit to Aberdeen. 
 The original proof of a weaker form of Theorem~\ref{T1} was presented in a three-hour seminar 
 at the Fields Institute in February 2016. I  would like to thank the audience, and to 
 George Elliott and Alessandro Vignati in particular,  for numerous sharp observations. 
 After the completion of the present paper Stuart White and Dan Voiculescu  pointed out that 
its results are related to Hadwin's asymptotic double commutant 
theorem  (\cite{hadwin1978asymptotic}, see also \cite{hadwin2011approximate} and
\cite{hadwin2012approximate}), and Martino Lupini pointed out that Theorem~\ref{T1} also holds 
in the nonunital case. I am indebted to the anonymous referee for several 
useful remarks. 
 Last, but not least, I would like to thank Leonel Robert for pointing out an error in an 
 early draft.

\section{Model theory of representations} 
We  expand the language of \cstar-algebras (\cite[\S 2.3.1]{FaHaSh:Model2}) to
   representations of \cstar-algebra. Reader's familiarity with, or at least easy access to, 
   \cite[\S 2]{FaHaSh:Model2} is assumed. 
A structure  in the expanded language $\calLrep$ is a \cstar-algebra together with its 
representation on a Hilbert space. As in \cite{FaHaSh:Model2}, the domains of  quantification on 
\cstar-algebra are $D_n$ for $n\in \bbN$ and are interpreted as  the $n$-balls. 
The domains of quantification on Hilbert space are~$D^H_n$ for $n\in \bbN$  and are also interpreted as the  $n$-balls. 
On all domains the metric is $d(x,y)=\|x-y\|$  (we shall denote both the operator norm on \cstar-algebra and 
the $\ell_2$-norm on Hilbert space by $\|\cdot\|$). 
As in \cite[\S 2.3.1]{FaHaSh:Model2}, for every $\lambda\in \bbC$ we have a 
unary function symbol $\lambda$ to be interpreted as multiplication by~$\lambda$. 
We also have binary function $+$ whose interpretation sends $D^H_m\times D^H_n$ to 
$D^H_{m+n}$. 
As the scalar product
 $(\cdot|\cdot)$ is definable from the norm via the polarization identity, 
 we shall freely use it in our formulas, with the understanding that~$(\xi|\eta)$  is an abbreviation 
 for $\frac 14 \sum_{j=0}^3 i^j \|\xi+i^j \eta\|$.

Language $\calLrep$ also contains a binary function symbol $\pi$ whose interpretation sends 
 $D_n\times D^H_m$ to $D^H_{mn}$ for all $m$ and $n$. It is 
interpreted as an action of $A$ on $H$.

Every variable is associated with a sort. In particular  variables
$x,y,z$ range over the \cstar-algebra and variables $\xi,\eta,\zeta$ 
range over the Hilbert space,  all of them 
decorated with subscripts when needed. 

We shall write $\bar x$ for a tuple $\bar x=(x_1,\dots, x_n)$ (with $n$ either 
clear from the context or irrelevant). 
\emph{Terms} come in two varieties. 
 On the \cstar-algebra side,  term is a noncommutative $^*$-polynomial in \cstar-variables. 
 On the Hilbert space side terms are linear combinations of 
 Hilbert space variables and expressions of the form $\pi(\alpha(\bar x))\xi$ 
where $\alpha(\bar x)$ is a term in the language of \cstar-algebras. 
\emph{Formulas}  are defined recursively.   
 Atomic formulas are   expressions of the form $\|t\|$ where $t$ is a term.

 The set of all formulas is the smallest set $\bbF$ containing all atomic formulas 
with the following properties. 
\begin{enumerate}
\item  [(i)] for every~$n$, all continuous $f\colon [0,\infty)^n\to [0,\infty)$ and all $\varphi_1, \dots, \varphi_n$ in $\bbF$
the expression $f(\varphi_1, \dots, \varphi_n)$ belongs to  $\bbF$, and 
\item [(ii)] if $\varphi\in \bbF$   and $x$, $\xi$ are variable symbols then each of  
 $\sup_{\|\xi\|\leq m} \varphi$,  $\inf_{\|\xi\|\leq m} \varphi$, 
   $\sup_{\|x\|\leq m} \varphi$ and $\inf_{\|x\|\leq m} \varphi$ 
    belongs to $\bbF$
 (see  \cite[\S 2.4]{FaHaSh:Model2} or \cite[Definition~2.1.1]{Muenster}). 
\end{enumerate}
Suppose $\pi\colon A\to B(H)$ is a representation of a \cstar-algebra $A$ on Hilbert space $H$. 
To $(A,H,\pi)$ we associate the natural metric structure $\cM(A,H,\pi)$ 
in the above  language. 

Suppose $\varphi(\bar x, \bar \xi)$ is a formula whose free variables are included among  
$\bar x$ and $\bar\xi$. 
If  $\pi\colon A\to B(H)$ is a representation of a \cstar-algebra on Hilbert space,  
$\bar a$ are elements of 
$A$ and $\bar \xi$ are elements 
of $H$,\footnote{Symbols $\xi,\eta, \zeta$,\dots{} denote both Hilbert space variables and vectors in Hilbert space due to the  font shortage; this shall not lead to a confusion.}
 then the \emph{interpretation} $\varphi(\bar a, \bar \xi)^{\cM(A,H,\pi)}$ 
is defined by recursion on the complexity of $\varphi$ in the obvious way (see \cite[\S 3]{BYBHU}).

\begin{prop} \label{L.Ax.1} Triples $(A,H,\pi)$ such that $\pi$ is a representation of $A$ on $H$ 
form an axiomatizable class. 
\end{prop} 

\begin{proof} 
As in \cite[Definition~3.1]{FaHaSh:Model2}, 
we need to define a $\calLrep$-theory $\bfTrep$ 
  such that 
that the category of triples $(A, H,\pi)$ where $\pi\colon A\to B(H)$ 
is a representation of a \cstar-algebra $A$ is equivalent to the category 
of metric structures that are models of $\bfTrep$, 
via the map 
\[
(A,H,\pi)\mapsto \cM(A, H,\pi). 
\]
We use 
the axiomatization of \cstar-algebras from \cite[\S 3.1]{FaHaSh:Model2}. 
In addition to the standard Hilbert space axioms, we need two axioms assuring that  
the interpretation of $D^H_n$ equals the $n$-ball of the underlying Hilbert space
for all $n$: 
\[
\sup_{\xi \in D_n}\|\xi\|\leq n
\]
and
 (writing  $s\dminus t:=\max(s-t,0)$), 
\begin{enumerate}
\item [(*)]  \label{Eq.DH}
$\sup_{\xi \in D_n} (1\dminus \|\xi\|)\inf_{\eta\in D_1} \|\xi-\eta\|$.
\end{enumerate}
 The standard axioms, 
\[
\pi(xy)\xi=\pi(x)\pi(y)\xi, \quad
\pi(x+y)\xi=\pi(x)\xi+\pi(y)\xi\quad
(\pi(x)\xi|\eta)=(\xi|\pi(x^*)\eta)
\]
are expressible as first-order sentences.\footnote{Our conventions are as described in \cite[p. 485]{FaHaSh:Model2}. 
In particular $\alpha(x,\xi)=\beta(x,\xi)$ 
is an abbreviation for  $\sup_{\xi\in D_n} \sup_\xi\|\alpha(x,\xi)-\beta(x,\xi)\|=0$, for all $n$.} 
The axioms described here comprise theory $\bfTrep$.

One needs to check that the category of models of $\bfTrep$  
is equivalent to the category 
of triples $(A,H,\pi)$. Every triple $(A,H,\pi)$ uniquely defines a model $\cM(A,H,\pi)$. 
Conversely, assume $\cM$ is a model of $\bfTrep$. 
The algebra~$A_{\cM}$ obtained from the first component of $\cM$  
is a \cstar-algebra by \cite[Proposition~3.2]{FaHaSh:Model2}. 
Also,  the linear space $H_{\cM}$ obtained from the second component of $\cM$
is a Hilbert space and the third component gives a representation $\pi_{\cM}$ of $A$ on~$H$. 

To see that this provides an equivalence of categories, we need to check that 
$\cM(A_{\cM},H_{\cM},\pi_{\cM})\cong \cM$ for every model $\cM$ of $\bfTrep$. 
We need to show that the domains on $\cM$ are determined by $A_{\cM}$ and $H_{\cM}$. 
The former was proved in  the second paragraph of 
 \cite[Proposition~3.2]{FaHaSh:Model2}, and the latter follows by~(*). 
 \end{proof} 
 
 Proposition~\ref{L.Ax.1} gives us full access to the model-theoretic toolbox, such as the   \L o\'s' theorem 
  (see  \S\ref{S.SR}) and the L\"owenheim--Skolem Theorem (\cite[Theorem~4.6]{FaHaSh:Model2}). 
From now on, we shall identify triple $(A,H,\pi)$ with the 
associated metric structure $\cM(A,H,\pi)$ and stop using the latter notation. 
We shall also write $\sup_{\|\xi\|\leq n}$ and $\inf_{\|\xi\|\leq n}$ instead of $\sup_{\xi\in D_n}$ 
and $\inf_{\xi\in D_n}$, respectively. 

\begin{lemma} \label{L.Ax.2} The following properties of a  representation $\pi$ of $A$ are axiomatizable: 
\begin{enumerate}
\item  $\pi$ is faithful  
\item $\pi$ is irreducible. 
\pushcounter
\end{enumerate}
\end{lemma} 

\begin{proof} We shall explicitly write the axioms for each of the  properties of~$\pi$. 
Fix a representation $\pi$.  It is faithful if and only if it is isometric, which can be expressed as  
\[
\sup_{\|x\|\leq 1} \inf_{\|\xi\|\leq 1}|\|x\|-\|\pi(x)\xi\||=0. 
\]
Representation $\pi$ 
is irreducible if and only if for all vectors $\xi$ and $\eta$ in~$H$ such that $\|\eta\|\leq 1$ and $\|\xi\|= 1$ 
the expression  $\|\eta-\pi(a)\xi\|$ can be made arbitrarily small when $a$ ranges over the unit ball of $A$.  In symbols 
\[
\sup_{\|\xi\|\leq 1}\sup_{\|\eta\|\leq 1}\inf_{\|x\|\leq 1} |\|\xi\|\dminus 1|\|\eta-\pi(x)\xi\|=0. 
\]
The interpretation of this sentence in $(A, H,\pi)$ is  0 if and only if $\pi$ is irreducible. 
\end{proof} 

Triple  $(D,\theta,K)$ is an \emph{elementary submodel} 
of $(B,\pi,H)$, and  $(B,\pi,H)$ is an \emph{elementary extension} of 
$(D,\theta,K$,  if $D\subseteq B$, $K\subseteq H$
$\theta(d)=\pi(d)\rs H$ for all $d\in D$, and   
 \[
 \varphi(\bar a)^{(D,\theta, K)}=\varphi(\bar a)^{(B,\pi,H)}
 \]
  for all formulas $\varphi$ 
 and all $\bar a$ in $(D,\theta, K)$ of the appropriate sort. 
Axiomatizable properties, such as being irreducible or faithful, 
transfer between elementary submodels and elementary extensions. 
Therefore the Downwards L\"owenheim--Skolem Theorem (\cite[Theorem~4.6]{FaHaSh:Model2}) and Lemma~\ref{L.Ax.2} together imply  e.g. that if $\varphi$ is a pure state of 
a nonseparable \cstar-algebra $B$ then $B$ is an inductive limit of 
separable subalgebras $D$ such that the restriction of $\varphi$ to~$D$ is pure.
This fact was proved in   \cite{AkeWe:Consistency} and its 
slightly more precise version  will be used in the proof of Lemma~\ref{L.**.2}.

Some other properties of representations (such as not being faithful) 
 are axiomatizable but we shall concentrate on proving Theorem~\ref{T1}.

\section{Saturation and representations} \label{S.SR} 
It was known to logicians since the 1960s that 
the two defining properties of ultraproducts  associated with nonprincipal ultrafilters on $\bbN$ in axiomatizable categories 
are \L o\'s' Theorem (\cite[Proposition~4.3]{FaHaSh:Model2}) 
and countable saturation (\cite[Proposition~4.11]{FaHaSh:Model2}). 
By the former, the diagonal embedding of metric structure $M$ into its ultrapower is 
elementary. 
More generally, 
if $\varphi(\bar x)$ is a formula and $\bar a(j)\in M_j$ are of the appropriate 
sort then 
\[
\varphi(\bar a)^{\prod_{\cU} M_j}=\lim_{j\to \cU} \varphi(\bar a(j))^{M_j}. 
 \]
In order to define countable saturation, 
we recall the notion of a type from the logic of metric structures (\cite[\S 4.3]{FaHaSh:Model2}).  
 A  \emph{closed condition} (or simply a \emph{condition}; we shall not need any other conditions) 
 is any expression of the form $\varphi\leq r$  for formula $\varphi$ and $r\geq 0$
 and \emph{type} is a set of conditions
 (\cite[\S 4.3]{FaHaSh:Model2}).   As every expression of the form $\varphi=r$ is equivalent to the condition $\max(\varphi, r)\leq r$
 and every expression of the form $\varphi\geq r$ is equivalent  to the condition $\min(0,r-\varphi)\leq 0$, we shall freely refer to such expressions as conditions. 
For $m$ and $n$ in $\bbN$ such that $m+n\geq 1$, an  \emph{$(m,n)$-type} is a type~$\bt$ such that all free variables occurring in  conditions of $\bt$ are among  $\{x_1, \dots, x_m\}\cup \{\xi_1, \dots, \xi_n\}$.

Given a structure $(A,H,\pi)$ and a subset $X$ of $A\cup H$, we expand the language 
$\calLrep$ by adding constants for the elements of $X$ (as in \cite[\S 2.4.1]{FaHaSh:Model2}). 
The new language is denoted 
 $(\calLrep)_X$. 
\cstar-terms in  $(\calLrep)_X$ 
are $^*$-polynomials in \cstar-variables and constants from $X\cap A$. 
Hilbert space terms are linear combinations of 
 Hilbert space variables,  constants in $X\cap H$,   and expressions 
 of the form $\pi(\alpha)\xi$ 
where $\alpha$ is a \cstar-term in the expanded language. 
The interpretation of  a $(\calLrep)_X$-formula is defined recursively in the natural way (see e.g.  \cite{Muenster}, 
paragraph after Definition~2.1.1). 
  
Type \emph{over} $X$ is a 
type in $(\calLrep)_X$.  
Such type is \emph{realized} in some elementary extension of $(A,H,\pi)$ if 
the latter contains a tuple satisfying all conditions from the type. 
A type  is \emph{consistent} if it is realized in some ultrapower of $(A,H,\pi)$, 
where the ultrafilter is taken over an arbitrary, not necessarily countable set. 
This is equivalent to the type being realized in some elementary extension of $(A,H,\pi)$. 
 
By \L o\'s' Theorem, type 
$\bt$ is consistent if and only every finite subset of $\bt$ is $\e$-realized in $(A,H,\pi)$ 
for every $\e>0$ (\cite[Proposition~4.8]{FaHaSh:Model2}).

A structure $(A,H,\pi)$ is said to be \emph{countably saturated} if every consistent type over  
a countable (or equivalently,
 norm-separable) set is realized in $(A,H,\pi)$. Ultraproducts associated with nonprincipal ultrafilters 
on $\bbN$ are always countably saturated (\cite[Proposition~4.11]{FaHaSh:Model2}). 
A standard transfinite back-and-forth argument shows that 
 a structure of density character $\aleph_1$ is countably saturated if and only if it is an 
 ultraproduct. 
(\emph{Density character} is the smallest cardinality of a dense subset.)

 In the case when  $A=B(H)$ we have $(B(H),H)^{\cU}=(B(H)^{\cU}, H^{\cU})$, 
 in particular $B(H)^{\cU}$  is  identified with a subalgebra of $B(H^{\cU})$. 
These two algebras are equal (still assuming $\cU$ is a nonprincipal ultrafilter on $\bbN$) 
if and only if $H$ is finite-dimensional. As a matter of fact, no projection 
$p\in B(H^{\cU})$ with a separable, infinite-dimensional,  range belongs to $B(H)^{\cU}$
(this is proved by a standard argument, 
see e.g. the last two paragraphs of the proof of \cite[Proposition~4.6]{FaHaSh:Model1}).

In the following,    $\pi$ will always be faithful and 
clear from the context and we shall identify~$A$ with $\pi(A)$ and suppress writing $\pi$. 
We shall therefore write $(A,H)$ in place of $(A,H,\id)$.

The following two lemmas are standard (they were used by Arveson 
in the proof of Corollary 2 on p 344 of \cite{Arv:Notes}) 
but we sketch the proofs for the   reader's convenience. 

\begin{lemma} \label{L.GNS} 
Suppose $A$ is a \cstar-algebra and $\varphi$ is a  
functional on $A$. Then there are a representation $\pi\colon A\to B(K)$ 
and vectors $\xi$ and $\eta$ in~$K$ such that $\varphi(a)=(\pi(a)\xi|\eta)$ for all $a\in A$. 
\end{lemma} 

\begin{proof} Let $\bar\varphi$ be the unique extension of $\varphi$ to 
a normal functional of the von Neumann algebra $A^{**}$. 
By Sakai's polar decomposition for normal linear functionals (see e.g. \cite[Proposition~3.6.7]{Pede:C*})
there exists a normal state $\psi$ of~$A^{**}$ and a partial isometry $v$ such that 
$\varphi(a)=\psi(av)$ for all $a\in A^{**}$. Let $\pi\colon A^{**}\to B(K)$ be the GNS representation 
corresponding to $\psi$. If $\eta$ is the corresponding cyclic vector and  $\xi=v\eta$, 
then the restriction of $\pi$ to $A$ is as required. 
\end{proof} 

\begin{lemma}\label{L.b}
Suppose $A$ is a proper unital subalgebra of $C=\cst(A,b)$. 
Then there exists a representation $\pi\colon C\to B(K)$ and a 
projection $q$ in $\pi(A)'\cap B(K)$ such that $[q,b]\neq 0$. 
\end{lemma} 

\begin{proof} 
By the Hahn--Banach separation theorem there exists a functional~$\varphi$ on $C$ of norm~$1$
such that $\varphi$ annihilates $A$ and $\varphi(b)=\dist(A,b)$.  
Let $\pi\colon C\to B(K)$, $\eta$ and $\xi$ be as guaranteed by 
Lemma~\ref{L.GNS}. Let $L$ be the norm-closure of~$\pi(A)\xi$. 
Since $A$ is unital $L\neq \{0\}$. As $0=\varphi(a)=(\pi(a)\xi|\eta)$ for all $a\in A$, 
$\eta$ is orthogonal to $L$ and therefore the projection $p$ to $L$ is nontrivial. 
Clearly $p\in \pi(A)'\cap B(K)$. 
Since $(\pi(b)\xi|\eta)=\varphi(b)\neq 0$, 
$\pi(b)$ does not commute with~$p$ and we therefore have 
 $q\in \pi(A)'\cap B(K)$ such that $\|[\pi(b),q]\|>0$. 
\end{proof} 

The proof of Theorem~\ref{T1} would be much simpler if 
Lemma~\ref{L.b} provided an irreducible representation. 
This is impossible in general  as the following example shows. 
Let $A$ be the unitization of the algebra of compact operators $\cK(H)$
on an infinite-dimensional Hilbert space and let $b$ be a projection in~$B(H)$ 
Murray-von Neumann equivalent to $1-b$. 
Then $C=\cst(A,b)$ has (up to equivalence) three irreducible representations. 
Two of those representations  annihilate $A$
and send $b$ to a scalar, and the third  representation is faithful and the image of 
$b$ is in the weak operator closure of the image of~$A$. 

 It is well-known that for a Banach space $X$ the second dual $X^{**}$ can be embedded into an  ultrapower of $X$ (\cite[Proposition~6.7]{heinrich1980ultraproducts}). In general, the second dual $A^{**}$ of a \cstar-algebra 
$A$  cannot be embedded into  an ultrapower of $A$ for at least two reasons. 
First, $A^{**}$ is a von Neumann algebra (\cite[\S III.5.2]{Black:Operator}) and it therefore
has real rank zero, while $A$ may have no nontrivial projections at all. 
 Since being projectionless is axiomatizable (\cite[Theorem~2.5.1]{Muenster}), 
  if $A$ is projectionless then \L o\'s's Theorem implies that $A^{\cU}$ is projectionless as well 
  and  $A^{**}$ cannot be  embeded  into it.    
The anonymous referee pointed out an another, much subtler, obstruction. 
In the context of Banach spaces, the embeddability of $X^{**}$ into $X^{\cU}$ is 
equivalent to a finitary statement, the so-called \emph{local reflexivity} of Banach spaces
the \cstar-algebraic version of which does not hold for all \cstar-algebras  
\cite[\S 5]{effros1985lifting}. In particular, for a large class of \cstar-algebras the diagonal 
embedding of $A$ into $A^{\cU}$ cannot be extended even to a unital  completely positive  map 
 from $A^{**}$ into $A^{\cU}$. 
The anonymous referee also pointed out that a result of J.~M.~G. Fell 
is closely related to  results of
the present section. It is a standard fact that a representation of a discrete group
is weakly contained in another representation of the same group if and only if it  
can be embedded into an ultrapower of the direct sum of infinitely many copies of the 
latter representation. 
In   \cite[Theorem~1.2]{fell1960dual} it was essentially proved that this equivalence carries over to arbitrary 
 \cstar-algebras.

All this said,  Lemma~\ref{L.**} below  
 is a poor man's 
 \cstar-algebraic 
 variant of the fact that Banach space~$X^{**}$ embeds into $X^{\cU}$. 
 As in \cite[3.3.6]{Pede:C*}, we say that two representations
 $\pi_1$ and $\pi_2$ of $A$ are said to be  \emph{equivalent} if the identity map on 
 $A$ extends to an isomorphism between 
 $\pi_1(A)''$ and $\pi_2(A)''$.

\begin{lemma}\label{L.**} Assume $(\prB, \prH)$ is an ultraproduct of  faithful 
irreducible representations of unital  \cstar-algebras and $C$ is a 
 unital separable subalgebra of $B^{\cU}$.  
\begin{enumerate}
\item 
If  $C\cap \cK(\prH)=\{0\}$  
then the induced representation of $C$ on $\prH$ is equivalent to 
the universal representation of $C$. 
\item In general, if 
\[
p=\bigvee\{q: q\text{ is a projection in }C\cap \cK(\prH)\}
\]
 then $p\in C'\cap B(\prH)$ and 
$c\mapsto (1-p) c$
is equivalent to the universal representation of $C/(C\cap \cK(\prH))$ on $(1-p)\prH$. 
\end{enumerate}
\end{lemma} 

\begin{proof} 
For a state $\psi$ on $C$ the $(0,1)$-type $\bt_\psi(\xi)$ of a vector $\xi$ implementing $\psi$ 
consists of all conditions of the form  $(a\xi|\xi)=\psi(a)$ for  $a\in C$ and $\|\xi\|=1$. 

(1) 
Fix a state $\psi$ on $C$. By Glimm's Lemma (\cite[Lemma~II.5.1]{Dav:C*}) 
 type $\bt_\psi$ is  
consistent with the theory of $(\prB, \prH)$.   
By the separability of~$C$ and countable saturation, there exists a unit vector 
$\eta\in \prH$ such that 
$\psi(c)=(c\eta|\eta)$ for all $c\in C$. Let~$L$ be the 
norm-closure of $C\eta$ in $\prH$. Then $L$ is 
a reducing subspace for~$C$ and the induced representation of $C$
on $L$ is spatially isomorphic to the  
 GNS representation of $C$ corresponding to $\psi$. 
 Since $\psi$ was arbitrary, by \cite[Theorem~3.8.2]{Pede:C*} this completes the proof. 

(2) For every $a\in C$ we have $pa\in C\cap \cK(\prH)$ and therefore $pa(1-p)=0$ 
Similarly $(1-p)ap=0$, and hence $p\in C'\cap B(\prH)$.
Let $p_n$, for $n\in \bbN$, be a maximal family of orthogonal projections in $C\cap \cK(\prH)$. 
It is countable by the separability of $C$ and $p=\bigvee_n p_n$. 
Let $\psi$ be a state of $C$ that annihilates $C\cap \cK(\prH)$. 
Let $\bt_\psi^+(\xi)$ be the type obtained from $\bt_\psi(\xi)$ by adding to it all 
conditions of the form 
$p_n\xi=0$
for  $n\in \bbN$. 
By Glimm's Lemma (as stated in \cite[Lemma~II.5.1]{Dav:C*}) 
the type $\bt_\psi^+(\xi)$ is consistent, and by the countable saturation 
we can find $\xi_1\in \prH$ that realizes this type. Then $p\xi_1=0$ and therefore $\xi_1\in (1-p)\prH$. 
Therefore every GNS representation of $C/(C\cap \cK(\prH))$ is spatially equivalent to a subrepresentation of $c\mapsto (1-p)c$, and by \cite[Theorem~3.8.2]{Pede:C*} 
 this concludes the proof.  
 \end{proof}

\section{Second dual and Day's trick} \label{S.Day} 
The natural embedding of a 
\cstar-algebra $B$ into its second dual $B^{**}$ is rarely elementary. 
For example, having real rank zero 
is axiomatizable (\cite[Theorem~2.5.1]{Muenster}) and~$B^{**}$, being a von Neumann algebra, has real rank zero  while
$B$ may  have no nontrivial projections at all.   
  However,  we shall see that there is a  restricted degree of elementarity between  $B$ and $B^{**}$, and it   
    will suffice for our purposes.

We shall consider the  language $(\calLrep)_B$ obtained by adding new constants
for parameters in $B$   (see \S\ref{S.SR}). 
Term~$\alpha(x)$ in the extended language is \emph{linear} if it is of the form 
\[
\alpha(x)=xa+bx
\]
 for some parameters $a$ and $b$. 
 
A \emph{restricted $B$-linear formula} is a formula of the form 
\begin{enumerate}
\item \label{Eq.rBlf}
$\max_{j\leq m} \|\alpha_j(x)-b_j\|+\max_{j\leq n}(r_j\dminus \|\beta_j(x)\|)$
\pushcounter
\end{enumerate}
where 
\begin{enumerate}
\popcounter
\item all $b_j$, for $1\leq j\leq m$ are parameters in $B$, 
\item all $r_j$, for $1\leq j\leq n$ are positive real numbers, 
\item all $\alpha_j$ for $1\leq j\leq m$ 
are linear terms with parameters in $B$, and
\item all $\beta_j$ for $1\leq j\leq n$  
are linear terms with parameters in $B$.  
\pushcounter
\end{enumerate}
%
%
Proof of the following is based on an application of the Hahn--Banach separation theorem 
first used by Day (\cite{day1957amenable};  see also \cite[Section 2]{Ell:SomeIII} for some  uses of this method in the theory of 
\cstar-algebras). 

\begin{lemma} \label{L.**.1} Suppose $B$ is a unital \cstar-algebra and 
\[
\gamma(x)=\max_{j\leq m} \|\alpha_j(x)-b_j\|+\max_{j\leq n}(r_j\dminus \|\beta_j(x)\|).
\]
is a restricted $B$-linear formula. 
Then the following are equivalent. 
\begin{enumerate}
\popcounter
\item
\label{I.1.1}  $\inf_{x\in B}\gamma(x)=0 $
\item \label{I.1.2} $\inf_{x\in B^{**}}\gamma(x)=0$. 
\pushcounter
\end{enumerate}
\end{lemma} 

\begin{proof} 
\eqref{I.1.1} implies  \eqref{I.1.2}
because  $B$ is isomorphic to a unital subalgebra of~$B^{**}$ 
and therefore
$\inf_{x\in B^{**}}\gamma(x)\leq \inf_{x\in B}\gamma(x)$. 

Assume  
 \eqref{I.1.2} holds. 
Let $a_j$ and $c_j$, for $j\leq n$, be such that 
 $\alpha_j(x)=a_jx+xc_j$. 
 For each $j$ we shall identify $\alpha_j$ with its interpretation, linear map from $B$ to $B$. 
 The second adjoint $\alpha_j^{**}\colon B^{**}\to B^{**}$ 
also satisfies $\alpha_j^{**}(x)=a_j x+x c_j$, hence~$\alpha_j^{**}(x)$ is the interpretation of the term $\alpha_j(x)$ in~$B^{**}$. 
The set 
\[
Z:=\langle \alpha_j(x): x\in B_{\leq 1} \rangle
\]
is, being an image of a convex set under a linear map, 
 a convex subset of~$B^m$ and by the Hahn--Banach theorem 
\[
Z_1:=B^m\cap \langle \alpha_j(x): x\in B^{**}_{\leq 1}\rangle
\]
is included in the norm-closure of $Z$. By \eqref{I.1.2} we have $(b_1,\dots, b_m)\in Z_1$. 

Fix $\e>0$ and let 
\[
\textstyle X_1:=\{x\in B_{\leq 1}:\max_{j\leq m} \| \alpha_j(x)-b_j\|\leq \e\}.
\] 
By the above this is a  convex subset of the unit ball of $B$ and (by using the Hahn--Banach separation theorem again) the weak$^*$-closure of $X_1$ in $B^{**}$  is equal to $\{x\in B^{**}_{\leq 1} :
\max_{j\leq m} \| \alpha_j(x)-b_j\|\leq \e\}$.

Let $c\in B^{**}_{\leq 1}$ be such that $\gamma(c)< \e$. Then $c$ belongs to the weak$^*$-closure of $X_1$. 
For each $j\leq n$ we have $\|\beta_j(c)\|> r_j-\e$. Fix a norming functional 
$\varphi_j\in B^*$ 
 such that $\|\varphi_j\|=1$ and  $\varphi_j(\beta_j(c))> r_j-\e$. 
Then 
\[
U:=\{x\in B^{**}: \varphi_j(\beta_j(x))> r_j-\e\text{ for all }j\}
\]
is a weak$^*$-open neighbourhood of $c$ and, as $c$ belongs to the weak$^*$-closure of~$X_1$, 
 $U\cap X_1$ is a nonempty subset of $B_{\leq 1}$. 
Any $b\in U\cap X_1$ satisfies $\gamma(b)< \e$. As $\e>0$ was arbitrary, this shows that 
\eqref{I.1.1} holds. 
\end{proof}


%

 In the following
 $A\subseteq \prB$ is identified with a subalgebra of $B(\prH)$. 

\begin{lemma} \label{L.**.2} 
Suppose $(B_j,H_j)$ is an irreducible representation of $B_j$ on~$H_j$ for $j\in \bbN$
and $A$ is a separable subalgebra of $\prB$. 
\begin{enumerate}
\item 
For every $b\in \prB$ we have 
$b\in (A'\cap B(\prH))'$ if and only if $b\in (A'\cap \prB)'$.  
Equivalently, 
\[
\textstyle (A'\cap B(\prH))'\cap \prod_{\cU} B_j=(A'\cap \prB)'\cap \prod_{\cU} B_j.
\]
\item $\WOT{A}\cap \prB=(A'\cap \prB)'$. 
\end{enumerate}
\end{lemma} 

\begin{proof} (1) Since $\prB\subseteq B(\prH)$ we clearly have 
$(A'\cap B(\prH))'\subseteq 
(A'\cap \prB)'$.  
In order to prove the converse inclusion, fix $b\in \prB$ and suppose that 
   there  exists $q\in A'\cap B(\prH)$ such that $\|[q,b]\|=r> 0$. We need to  find  
 $d\in A'\cap \prB$ satisfying $[d,b]\neq 0$. 
 
Consider the $(1,0)$-type $\bt(x)$ consisting of all  conditions of the form 
\[
\|[x,b]\|\geq r\quad\text{and}\quad
[x,a]=0
\]
 for  $a\in A$.  
This type is satisfied in $B(\prH)$ by $q$. 
Since all formulas in~$\bt(x)$  are quantifier-free, 
their interpretation is unchanged when passing to a larger  algebra.

Fix a finite subset of $\bt(x)$ and let $F\subseteq A$ be the  set of parameters occurring in this subset. 
Then 
\[
\gamma_F(x):=\inf_x \max_{a\in F} \|[x,a]\|+(r\dminus \|[x,b]\|)
\]
is a restricted $\prB$-linear formula.  Since $A$ is separable, we can find 
   projection $p$ in $\cst(A,q)'\cap B(\prH)$ with separable range 
such that $q_1:=pq$ 
satisfies $\|[q_1,b]\|=r$.  
To find $p$ as required, take a separable elementary submodel $(C,H_0)$ of $(B(\prH), \prH)$
such that $A\subseteq C$ and let $p$ be the projection to $H_0$.

By the Downward L\"owenheim--Skolem Theorem (\cite[Theorem~4.6]{FaHaSh:Model2})
 there exists  
 a separable elementary submodel $(D,K)$ 
of $(\prB, \prH)$
such that $\cst(A,b)\subseteq D$ and  the range of $p$ is included in $K$. 
Part (2) of  Lemma~\ref{L.Ax.2} and \L o\'s' Theorem imply that 
      $\WOT{\prB}=B(\prH)$ 
      and   (with $p_K$ denoting the projection to $K$) 
that $\WOT{p_K D p_K} =B(p_K \prH)$. 
We can therefore identify~$p_K$ with a minimal central projection  in~$D^{**}$. 
Via this identification  we have 
$q_1\in D^{**}$. 
    Since $\gamma_F(q_1)=0$,     Lemma~\ref{L.**.1} 
implies 
 $ \inf_{x\in D, \|x\|\leq 1}\gamma_F(x)=0$ 
and 
 $\inf_{x\in \prB, \|x\|\leq 1}\gamma_F(x)=0$
 (since $\gamma_F$ is  quantifier-free). 

 Since $F$ was an arbitrary finite subset of $A$  
  type    $\bt(x)$ is consistent with the theory of $\prB$. 
Since $A$ is separable, by the countable saturation there exists $d\in A'\cap \prB$ satisfying 
$\|[d,b]\|\geq r$. 

(2) By the von Neumann bicommutant theorem $\WOT{A}=(A'\cap B(\prH))'$ and 
therefore (1) implies $\WOT{A}\cap \prB=(A'\cap \prB)'$. 
\end{proof}

\section{Proof of Theorem~\ref{T1}}
Suppose $(B_j,H_j)$ is a faithful 
 irreducible representation of $B_j$ on~$H_j$ for $j\in \bbN$ 
and  $A$ is  a separable subalgebra of $\prB$. 
By  Lemma~\ref{L.Ax.2}
  $(\prB, \prH)$ is an irreducible faithful representation of 
$\prB$. 

By (2) of Lemma~\ref{L.**.2}  we have $\WOT{A}\cap \prB=(A'\cap \prB)'$. 
Since $A\subseteq (A'\cap \prB)'$, it remains to prove $(A'\cap \prB)'\subseteq A$. 
Fix $b\in \prB$ such that $r:=\dist(b,A)>0$. 
 By (1) of Lemma~\ref{L.**.2} it suffices to find $d\in A'\cap B(\prH)$ such that $[d,b]\neq 0$. 
Let 
\[
C:=\cst(A,b).
\]

\begin{lemma} \label{L.Key} With $\prB$, $A$, $b$, $C$ and $r$ as above, 
there exists a representation $\pi\colon  C/(C\cap \cK(\prH))\to B(K)$ 
and $q\in \pi(A)'\cap B(K)$ such that $[q,\pi(b)]\neq 0$. 
\end{lemma} 

Since the proof of 
Lemma~\ref{L.Key} is on the long side, 
let us show how it completes the proof of Theorem~\ref{T1}. 
Lemma~\ref{L.**} implies that if 
\[
\textstyle p=\bigvee\{q: q\text{ is a projection in }C\cap \cK(\prH)\}
\]
 then 
$p\in C'\cap B(\prH)$ and 
$c\mapsto (1-p) c$
is equivalent to the universal representation of $C/(C\cap \cK(\prH))$ on $(1-p)\prH$. 
Therefore $q$ as in the conclusion of Lemma~\ref{L.Key} can be found in $A'\cap B(\prH)$, 
implying $b\notin (A'\cap B(\prH))'$. By Lemma~\ref{L.**.2} this implies $b\notin (A'\cap \prB)'$,
reducing the proof of Theorem~\ref{T1}  to the following.

\begin{proof}[Proof of Lemma~\ref{L.Key}]
An easy  special case may be worth noting. If  $C\cap \cK(\prH)=\{0\}$ then  
 Lemma~\ref{L.b} implies the  existence of   
 a representation 
 $\pi\colon  C\to B(K)$
and $q\in \pi(A)'\cap B(K)$ such that $[q,\pi(b)]\neq 0$.

In the general case, 
let  $q_n$, for $n\in J$, be an enumeration of a maximal orthogonal set of minimal projections in $A\cap \cK(\prH)$. 
 The index-set $J$  is countable (and possibly finite or even empty) since $A$ is separable. Let $p_n:=\bigvee_{j\leq n} q_j$. 

Suppose for a moment that there exists $n$ such that $p_n b p_n\notin A$. 
Since the range of $p_n$ is finite-dimensional, by von Neumann's Bicommutant Theorem 
 (\cite[\S I.9.1.2]{Black:Operator}).
 and  
the  Kadison Transitivity Theorem   (\cite[Theorem II.6.1.13]{Black:Operator})  
there exists $d\in A'\cap B(p_n\prH)$ such that $[d,b]\neq 0$.
Lemma~\ref{L.**.2} now implies  
 $p_n b p_n\notin (A'\cap \prB)'$ and $b\notin  (A'\cap \prB)'$. 

We may therefore assume $p_n b p_n\in A$ for all $n$. 
Let $p=\bigvee_n p_n$.  Lemma~\ref{L.**}  (2) implies 
 $p\in A'\cap B(\prH)$, 
and we may therefore assume $[b,p]=0$. Since $C=\cst(A,b)$ this implies $p\in C'\cap B(\prH)$. 
Since $p_n b p_n \in A$ for all $n$ we have $A\cap\cK(\prH)pCp\cap \cK(\prH)$. 
If $c\in C$ then for every $n$ we have $p_n c (1-p)=0$ and similarly 
$(1-p) c p_n=0$. Since the sequence $p_n$, for $n\in \bbN$, is an approximate unit for $A\cap \cK(\prH)$, 
the latter is an ideal 
of $C$. 
Let $\theta\colon C\to C/(A\cap \cK)$ be the quotient map. We claim that 
$\dist(\theta(b), \theta(A))=\dist(b,A)>0$.   
Fix $a\in A$. We need to show that $\|\theta(a-b)\|\geq r$. 

Consider the $(0,1)$-type $\bt(\xi)$ consisting of  all conditions of the form 
\[
\|\xi\|=1,\quad  \|(a-b)\xi\|\geq r,  \quad p_n \xi=0, 
\]
for  $n\in J$. 
To see that this type is consistent fix a finite $F\subseteq J$. Let $m\geq \max(F)$
and  
\[
a':=(1-p_m)a(1-p_m) + p_m b p_m. 
\]
As both summands belong to $A$, $a'\in A$ and therefore $\|a'-b\|\geq r$. 
Fix $\e>0$. 
If $\xi\in \prH$ is a vector of norm $\leq 1$  such that $\|(a'-b)\xi\|>r-\e$ then $\xi'=(1-p_m)\xi$ 
 has the same property since $(a'-b)p_m=0$. Since $\e>0$ was arbitrary,  
 $\bt(\xi)$ is consistent. By the countable saturation there exists a unit vector $\xi\in \prH$
which realizes $\bt(\xi)$. Since $p_n \xi=0$ for all $n$ we have $p\xi=0$ and 
therefore $\|\theta(a-b)\|\geq \|(1-p)(a-b)(1-p)\|\geq r$. 
Since $a\in A$ was arbitrary, we conclude that $\dist(\theta(b), \theta(A))=r$.

Suppose for  a moment that 
  $(1-p)C(1-p)\cap \cK(\prH)=\{0\}$.  By (2) of Lemma~\ref{L.**} 
the  representation 
\[
\textstyle C\ni c\mapsto (1-p)c\in B((1-p)\prH)
\]
 is equivalent to 
the universal representation 
of $C$.  Hence by  Lemma~\ref{L.b} 
 we can find $d\in (1-p)(A'\cap B(\prH))$ that does not commute with $b$, and by the above 
 this concludes the proof in this case. 

We may therefore assume that $(1-p)C(1-p)\cap \cK(\prH)\neq \{0\}$. 
By the spectral theorem for self-adjoint compact operators and continuous functional calculus, 
there exists a nonzero projection $q\in (1-p) C(1-p)$ of finite rank. 
Fix  $c\in C$ such that $(1-p)c(1-p)=q$.

By Lemma~\ref{L.**.2} it suffices to find 
 $q\in A'\cap (1-p)B(\prH)(1-p)$ such that $[q,c]\neq 0$. 
Suppose otherwise, that $c\in (A'\cap \prB)'$. 
(2) of  Lemma~\ref{L.**.2}  implies that $c\in \WOT{A}$.
By the Kaplansky Density Theorem (\cite[Theorem I.9.1.3]{Black:Operator}) 
 there is a net of positive contractions in $A$ converging to $c$ in the weak operator topology. 
By the continuous functional calculus and the Kadison Transitivity Theorem (\cite[Theorem II.6.1.13]{Black:Operator})  
we may choose this net among the members of 
 \[
Z:= \{a\in A_+: \|a\|=1, q a q=q\}. 
 \]
 Consider the $(0,1)$-type  $\bt_1(\xi)$ consisting of all conditions of the form 
\[
\|\xi\|=1,\quad a\xi=\xi,\quad q\xi=0, \quad p_n \xi=0
 \]
 for  $n\in \bbN$ and $a\in Z$. 

 We  claim that  $\bt_1(\xi)$ is consistent.  
 Fix  $\e>0$ and $a_1, a_2,\dots, a_n$ in $Z$. 
 Let 
 \[
 a:=a_1 a_2\dots a_{n-1} a_n a_{n-1} \dots a_2 a_1. 
 \]
 Then $a\in Z$ and $q\leq a$. By the choice of $p$ the operator 
    $(1-p)(a-s)_+$ is not compact for any $s<1$. 
Therefore  there exists a unit vector 
$\xi_0$ in $(1-p-q)\prH$ such that $\|\xi_0-a\xi_0\|$ is arbitrarily small. 
By the countable saturation there exists a unit vector $\xi_1\in (1-(p+q))\prH$ such that 
$a\xi_1=\xi_1$. As each $a_j$ is a positive contraction, we have $a_j\xi_1=\xi_1$ for $1\leq j\leq n$.  
 Since $a_1,\dots, a_n$ was an arbitrary subset of $Z$, this shows that $\bt_1(\xi)$ 
is consistent. 

Since $Z$ is separable, by the countable saturation
there exists $\xi\in \prH$ realizing $\bt_1(\xi)$. 
Then $\xi$ is a unit vector   in  $(1-(p+q))\prH$ such that $a\xi=\xi$ for all $a\in Z$. 
As $c\xi=0$, 
 this  contradicts   $c$ being in the weak operator topology closure of $Z$. 

Therefore there exists  $q\in A'\cap (1-p)B(\prH)(1-p)$ such that $[q,c]\neq 0$. 
Since $c\in C=\cst(A,b)$ we have  $[q,b]\neq 0$, and this concludes the proof. 
\end{proof}

\section{Concluding remarks}

In the following infinitary form  of the  Kadison Transitivity Theorem 
 $p_K$ denotes projection to a closed subspace $K$ of $\prH$. 

\begin{prop}\label{L.Kadison+} 
Assume $(\prB, \prH)$ is an ultraproduct of faithful and irreducible 
representations  of unital \cstar-algebras. 
Also assume $K$ is a separable closed subspace of $\prH$
and $T\in B(K)$. 
\begin{enumerate}
\item There exists $b\in \prB$ such that $\|b\|=\|T\|$ and $p_K b p_K=T$. 

\item If $T$ is self-adjoint (or positive, or unitary in $B(K)$) then $b$ can be chosen 
to be self-adjoint (or positive, or unitary in $B(\prH)$).   
\end{enumerate}
\end{prop} 

\begin{proof} (1) is a consequence of the Kadison transitivity theorem and countable saturation of
the structure $(\prB, \prH)$. Let $p_n$, for $n\in \bbN$, be an increasing sequence of finite-dimensional projections 
converging to $p_K$ in the strong operator topology and let $a_n$, for $n\in \bbN$, be a dense subset of $A$.  
We need to check that the type $\bt(x)$ consisting of all conditions of the form  
\[
\|p_n (x-T)p_n\|=0, 
 \qquad \|x\|=\|T\|
\]
for $n\in \bbN$ 
is consistent. 
Since the representation of $\prB$ on $\prH$ is irreducible by Lemma~\ref{L.Ax.2}, 
every finite subset of $\bt(x)$  is consistent by the Kadison Transitivity Theorem.  
We can therefore find $b\in \prB$ that satisfies $\bt(x)$  and therefore $p_Kbp_K=T$ and $\|b\|=\|T\|$. 

  (2)  If $T$ is self-adjoint, add the condition $x=x^*$ to $\bt(x)$. 
  By    
  \cite[Theorem~2.7.5]{Pede:C*}
the corresponding type is consistent, and the assertion again follows by countable saturation. The 
case when $T$ is a unitary also uses   \cite[Theorem~2.7.5]{Pede:C*}. 
\end{proof}

An important  consequence of full Voiculescu's theorem is that 
any two  unital representations $\pi_j\colon A\to B(H)$ 
of a separable unital \cstar-algebra $A$ on $H$ 
such that $\ker(\pi_1)=\ker(\pi_2)$ and $\pi_1(A)\cap \cK(H)=\pi_2(A)\cap \cK(H)=\{0\}$
are approximately unitarily equivalent (\cite[Corollary~1.4]{voiculescu1976non}). 
The analogous statement is in general false for  the ultraproducts. 
Let  $B_n= M_n(\bbC)$ for $n\in \bbN$ and let $A=\bbC^2$. 
Group $K_0(\prod_{\cU} M_n(\bbC))$ is isomorphic to $\bbZ^{\bbN}$ with the 
natural ordering and the identity function $\id$ as the order-unit. 
Every unital representation of~$A$ corresponds to an element of 
this group that lies between $0$ and $\id$, and there are~$2^{\aleph_0}$ inequivalent representations. 
Also, 
$K_0(\prod_{\cU} M_n(\bbC))$ is isomorphic to the ultraproduct $\prod_{\cU}\bbZ $
and $2^{\aleph_0}$  of these extensions remain inequivalent even after passing to the ultraproduct.

We return to G.~K. Pedersen's question 
 (\cite[Remark~10.11]{Pede:Corona}),  whether a bicommutant theorem 
$(A'\cap M(B)/B)'=A$ is true for a separable unital subalgebra $A$ of a corona $M(B)/B$ of a $\sigma$-unital \cstar-algebra $B$? 
A simple and unital  \cstar-algebra $C$ is \emph{purely infinite} if for every nonzero $a\in C$ there are $x$ and $y$ 
such that $xay=1$.

\begin{question} \label{Q2} Suppose $C$ is a unital, simple, purely infinite, and separable
 and $A$ is a unital subalgebra 
of $C$. Is $(A'\cap C^{**})'\cap C=A$? 
\end{question}

Let us prove that a positive answer to Question~\ref{Q2} would imply a positive answer to Pedersen's question. 
If $A$ is a separable and unital subalgebra of $M(B)/B$ and $b\in (M(B)/B)\setminus A$, 
then there exists a separable elementary submodel $C$ of $M(B)/B$ containing $b$. 
By \cite{lin2004simple}, $M(B)/B$ is simple if and only if it is purely infinite, 
Since being simple and purely infinite is axiomatizable (\cite[Theorem~2.5.1]{Muenster}), 
$C$ is simple and purely infinite. If $(A'\cap C^{**})'\cap C=A$ then 
Proposition~\ref{P.Key} below  implies that there exists
$d\in A'\cap M(B)/B$ such that $[d,b]\neq 0$.

\begin{prop} \label{P.Key} Suppose $B$ is a \cstar-algebra, $A$ is a separable subalgebra of~$B$, $b\in B$ and $r\geq 0$. 
If $B$ is an ultraproduct or a corona of a $\sigma$-unital, non-unital \cstar-algebra then 
\[
\sup_{d\in (A'\cap B)_+, \|d\|\leq 1} \|[d,b]\| 
=\sup_{d\in (A'\cap B^{**})_+, \|d\|\leq 1}  \|[d,b]\|. 
\]
\end{prop} 

\begin{proof}
The only property of $B$ used in the proof of Proposition~\ref{P.Key} (given at the end of this section) 
is that of being countably degree-1 saturated (\cite[Theorem~1]{FaHa:Countable}). 
Since $B\subseteq B^{**}$, it suffices to prove  `$\geq$' in the above inequality.   
Suppose  $b\in B$ and $d\in (A'\cap B^{**})_+$ are  such that 
$\|d\|=1$ and $r\dminus \|[b,d]\|$. Consider the  type $\bt(x)$ consisting 
of conditions $\|x\|=1$, $x\geq 0$, $\|xb-bx\|\geq r$, and 
$\|[x,a]\|=0$ for $a$ in a countable dense subset of $A$, 
This is a countable degree-1 type. If $\phi_j=0$, for $j<n$, is a finite subset of $\bt(x)$ then 
$\gamma(x):=\max_{j<n} \phi_j(x)$ is a restricted $B$-linear formula and Lemma~\ref{L.**} 
implies that it is approximately satisfied in $B$. By the countable degree-1 saturation of $B$
 (\cite[Theorem~1]{FaHa:Countable}) we can find a realization $d'$ of $\bt(x)$ in $B$. 
 Clearly $d'\in (A'\cap B)_+$, $\|d'\|=1$, and $\|[d',b]\|\geq r$, completing the proof. 
 \end{proof}

Some information on a special case of Pedersen's conjecture can also be found in \cite{kucerovsky2007relative}.

\providecommand{\bysame}{\leavevmode\hbox to3em{\hrulefill}\thinspace}
\providecommand{\MR}{\relax\ifhmode\unskip\space\fi MR }
\providecommand{\MRhref}[2]{%
  \href{http://www.ams.org/mathscinet-getitem?mr=#1}{#2}
}
\providecommand{\href}[2]{#2}

\end{document}